\documentclass[reqno,10pt]{amsart}

\usepackage{amsmath,amsthm,amssymb}

\theoremstyle{plain}
\newtheorem{theorem}{Theorem}[section]
\newtheorem{corollary}[theorem]{Corollary}
\newtheorem{proposition}[theorem]{Proposition}
\newtheorem{lemma}[theorem]{Lemma}

\theoremstyle{definition}
\newtheorem{definition}[theorem]{Definition}
\newtheorem{remark}[theorem]{Remark}

\newcommand{\C}{\mathbb{C}}

\newcommand{\N}{\mathbb{N}}

\newcommand{\F}{\mathcal{F}}

\newcommand{\norm}[1]{\left\lVert#1\right\rVert}
\newcommand{\abs}[1]{\left | {#1}\right | }
\newcommand{\all}[2]{ \{\, {#1} \, : \, {#2} \, \} }

\DeclareMathOperator{\rank}{rank}

\DeclareMathOperator{\Hdist}{Hdist}

\begin{document}

\title[Spectral distance of trace class operators]{Explicit upper bounds for the spectral distance of two trace class operators}

\author[O.F.~Bandtlow]{Oscar F.~Bandtlow}
\address{%
Oscar F.~Bandtlow\\
School of Mathematical Sciences\\
Queen Mary University of London\\
London E1 4NS\\
UK.
}
\email{o.bandtlow@qmul.ac.uk}

\author[A.~G\"{u}ven]{Ay\c{s}e G\"{u}ven}
\address{%
Ay\c{s}e G\"{u}ven\\
School of Mathematical Sciences\\
Queen Mary University of London\\
London E1 4NS\\
UK.
}
\email{a.guven@qmul.ac.uk}



\date{September 27, 2014}

\begin{abstract}
Given two trace class operators $A$ and $B$ on a separable Hilbert
space we provide an upper bound for the Hausdorff
distance of their spectra involving only the distance of $A$ and $B$
in operator norm and the singular values of $A$ and
$B$. By specifying particular asymptotics of the singular values 
our bound reproduces or
improves existing bounds for the spectral distance. 
The proof is
based on lower and upper bounds for determinants of trace class
operators of independent interest.

\end{abstract}

\maketitle

\section{Introduction}
Given an arbitrary compact operator $A$ on a separable Hilbert space an 
interesting question of practical importance is to determine its
spectrum $\sigma(A)$. One way of tackling it numerically is to reduce the 
infinite-dimensional problem to a finite-dimensional one by
manufacturing a sequence of finite rank operators $(A_k)_{k=1}^\infty$ 
converging to $A$ in operator norm and using the fact that the eigenvalues of $A_k$,
converge to the eigenvalues of $A$. In practice, one has to stop
after a finite number of steps. If there is interest in 
error estimates, the problem arises how to explicitly
bound the distance of the spectrum of a finite rank approximant $A_k$ to  
the spectrum of $A$ in a suitable sense. A popular choice is 
the Hausdorff metric, the definition of which we briefly recall. 

For $z\in \C$ and a compact subset $\sigma\subset \C$ let  
\[ d(z,\sigma)=\inf_{\lambda\in \sigma}\abs{z-\lambda} \] 
denote the distance of $z$ to $\sigma$. Given two compact subsets 
$\sigma_1, \sigma_2 \subset \mathbb C$ their  
\emph{Hausdorff distance} 
is given by 
\[
\Hdist(\sigma_1, \sigma_2) 
 = \max\{ \hat d(\sigma_1, \sigma_2), \hat d(\sigma_2, \sigma_1)\}
\]
where 
\[
\hat d(\sigma_1, \sigma_2) = \sup_{\lambda \in\sigma_1} d(\lambda, \sigma_2)\,.
\] 
It is not difficult to see that the Hausdorff distance is a metric on the set of 
compact subsets of $\C$. 

The finite-dimensional prototype of an explicitly computable bound for
the Hausdorff distance of the spectra of two $n\times n$ matrices $A$
and $B$, also known as the \emph{spectral distance of $A$ an $B$} in
this context, is due to Elsner \cite{4}, sharpening 
various earlier bounds (see \cite[Chapter~VIII]{bhatia} for a discussion) 
beginning with Ostrowski \cite{ostrowski}, 
and reads as follows
\begin{equation}
\label{eq:elsner}
\Hdist(\sigma(A), \sigma(B)) \leq 
(\norm{A}+\norm{B})^{1-{\frac{1}{n}}} \norm{A-B}^{\frac{1}{n}}\,,
\end{equation} 
where $\norm{\cdot}$ denotes the operator norm corresponding to the Euclidean 
norm on $\C^n$. Note that if the operator norm is 
taken with respect to any other 
(non-Euclidean) norm on $\C^n$, then (\ref{eq:elsner}) does not 
hold (see \cite{ransford}). 

Unfortunately, a straightforward extension of this formula to the
in\-fi\-nite-di\-men\-si\-o\-nal 
setting by simply passing to the limit fails,
due to the presence of the exponent $1/n$. However, 
versions of (\ref{eq:elsner}) hold 
for $A$ and $B$ algebraic elements of a Banach algebra of degree less than $n$ 
(see \cite{chenetal}). 

Infinite-dimensional versions of (\ref{eq:elsner}) have so far been 
obtained essentially by one method going back to Henrici \cite{henrici} which   
relies on writing a given compact operator 
as a perturbation of a quasi-nilpotent operator by a normal operator with 
the same spectrum as the original operator. Using this approach upper 
bounds for the spectral distance of Schatten class operators have been obtained by
Gil' in a series of papers (see \cite{33} and references therein for an 
overview) and one of the authors (see \cite{1,10}). 
Similar bounds, less sharp, 
but valid more generally for operators in symmetrically normed 
ideals are due to Pokrzywa \cite{71}, using the same method. For
recent extensions of the results in \cite{33} to unbounded operators
with inverses in the Schatten classes see \cite{gil12, gil13}.

This article grew out of an attempt to bypass the Henrici argument and instead 
make Elsner's delightful determinant based proof of 
(\ref{eq:elsner}) work in the infinite-dimensional setting. By considering determinants and 
keeping track of the singular values of the operators involved we arrive at a version of 
Elsner's formula~(\ref{eq:elsner}) which holds for trace class operators and which 
yields new bounds even in finite dimensions.

It is organised as follows. In Section~\ref{sec:2} we provide some background on trace class
operators $A$ and the corresponding determinants $\det(I-A)$.  
In the following section (Section~\ref{sec:3}) we first derive 
a lower bound for $\det(I-z^{-1}A)$, where $A$ is trace class, 
expressible solely in
terms of the distance of $z$ to the spectrum of $A$ and the singular
values of $A$. This bound is of independent interest and does not seem
to have appeared in the literature yet. Next, for $z$ an eigenvalue of 
a bounded operator $B$, we derive an upper bound for $\det(I-z^{-1}A)$
in terms of the distance of $z$ to the spectrum of $A$ and
$\norm{A-B}$. In Section~\ref{sec:4} we combine the upper and lower
bounds to obtain our main result. 
We finish (Section~\ref{sec:5}) by comparing our bound to Elsner's and show 
that it reduces to or improves the bounds in \cite{1,10}.

\section{Trace class operators and their determinants}
\label{sec:2}
In this section we briefly recapitulate some facts about trace class operators 
and their determinants, which will be crucial for what is to
follow.

Let $H$ be a separable Hilbert space with
scalar product $(\cdot , \cdot )$. 
We write $L(H)$ for the Banach space of bounded linear operators
on $H$ equipped with the uniform operator norm $\norm{\cdot}$. 
For $A\in L(H)$ the spectrum and the set of eigenvalues will be denoted
by $\sigma(A)$ and $\sigma_p(A)$, respectively, while the resolvent
set will be denoted by $\rho(A)$. 

If $A\in L(H)$ is a compact operator, we use 
 $(\lambda_n(A))_{n\in \N}$ to denote its
eigenvalue sequence, counting algebraic multiplicities and ordered by
decreasing modulus so that 
\[ \abs{\lambda_1(A)}\geq \abs{\lambda_2(A)} \geq \cdots \]
If $A$ has only finitely many non-zero eigenvalues, we set
$\lambda_n(A)=0$ for $n>N$, where $N$ denotes the number of non-zero
eigenvalues of $A$. For $n\in \N$, the $n$-th singular value of $A$ is given by 
\[ s_n(A)=\sqrt{\lambda_n(A^*A)} \quad (n \in \N)\,, \]
where $A^*$ denotes the Hilbert space adjoint of $A$. 

Eigenvalues and singular values satisfy a number of inequalities known
as Weyl's inequalities. The most important of them is the
multiplicative Weyl inequality (see, for example, 
\cite[Chapter~VI, Theorem~2.1]{3})  
\begin{equation}\label{pp}
\prod_{k=1}^{n}|\lambda_k(A)| \leq \prod_{k=1}^{n}s_k(A)  \quad
(\forall n \in \N)\,.
\end{equation}
The following additive version follows from the
multiplicative one 
(see, for example, \cite[Chapter~VI, Corollary~2.3]{3}) 
\begin{equation}\label{v}
\sum_{k=1}^{n} |\lambda_k(A)| \leq \sum_{k=1}^{n}s_j(A)  \quad
(\forall n\in \N)
\end{equation} 
and so does the following (see, for example, \cite[Chapter~VI, Corollary~2.5]{3}) 
\begin{equation}
\label{vv}
\prod_{k=1}^{n}(1+r|\lambda_k(A)|) \leq \prod_{k=1}^{n}(1+rs_k(A))
\quad (\forall r\geq 0, \forall n \in \N)\,.
\end{equation}

We now recall the notion of a trace class operator as an element of
the set   
\[
S_1(H)= \all{A \in L(H)}{\text{$A$ compact and 
 $\norm{A}_{1}= \sum_{k=1}^{\infty}s_k(A) < \infty$}}\,.
\]
It turns out that $\norm{\cdot}_{1}$ is a norm on $S_1(H)$ turning it
into a Banach space (see, for example, 
\cite[Chapter~VI, Corollary~3.2 and Theorem~4.1]{3}). 

With any a trace class operator $A\in S_1(H)$ it is possible to
associate a determinant given by 
\[ \det(I - A) = \prod_{k=1}^\infty(1- \lambda_k(A))\,,\]  
which, in view of (\ref{vv}) and the summability of the singular
values, is well-defined and satisfies the bound 
\[ \abs{\det(I-A)} \leq \prod_{k=1}^\infty(1+\abs{\lambda_k(A)}) 
   \leq \prod_{k=1}^\infty(1+s_k(A)) 
\leq \exp \left ( \norm{A}_1\right ) \,. \]
In particular, this implies that $z\mapsto \det(I-zA)$ is an entire
function of exponential type zero (see, for example, 
\cite[Chapter~VII, Theorem~4.3]{3}).

\section{Upper and lower bounds for determinants}
\label{sec:3}
In this section we shall derive both upper and lower bounds for the
determinant of a trace class operator $A$ expressible in terms of the
function $F_A:[0,\infty)\to [0,\infty)$ given by 
\begin{equation}
\label{eq:FAdef}
F_A(r)=\prod_{k=1}^\infty(1+r s_k(A))\,,
\end{equation}
which is well-defined, real-analytic and strictly monotonically increasing. 

Before deriving these bounds we briefly recall two more facts about
singular values. The first is a useful alternative characterisation in
terms of approximation numbers (see, for example, \cite[Chapter VI, Theorem~1.5]{3})
\begin{equation}
\label{eq:snisan}
s_n(A) = \inf \all{\norm{A-F}}{F\in L(H),\,\rank(F)<n} \quad
  (\forall n \in \N)\,.
\end{equation}
Using this characterisation the second fact follows:
if $G$ is a closed subspace of $H$ with orthogonal projection
$P$ we have (see, for example, 
\cite[Chapter VI, Corollary~1.4]{3})
\begin{equation}
  \label{eq:snpag}
   s_n(PA|_{P(H)}) \leq s_n(A) \quad (\forall n \in \N)\,.
\end{equation}

By definition of the determinant of a trace class oprator $A$, the
function $z\mapsto \det(I-z^{-1}A)$ has zeros at the non-zero
eigenvalues of $A$. The following proposition provides a new lower bound
for the behaviour of $z\mapsto \det(I-z^{-1}A)$ in the vicinity of a
zero solely in terms of the function $F_A$ defined in (\ref{eq:FAdef})
and the distance of $z$ to the spectrum of $A$. 

\begin{proposition}
\label{prop:detl}
Let $A\in S_1(H)$. Then, for any $z\in \rho(A)$ with $z\neq 0$, we have  
  \[ \abs{\det(I-z^{-1}A)}^{-1}\leq F_A\left ( \frac{1}{d(z,\sigma(A))} \right )\,.\]
\end{proposition}
\begin{proof}
For brevity, write $\lambda_k=\lambda_k(A)$ and $s_k=s_k(A)$. Then,
for $z\in \rho(A)$ with $z\neq 0$, we have                        
\begin{multline*}
  \abs{\det(I-z^{-1}A)}^{-1} 
   = \prod_{k=1}^{\infty} 
      \abs{ \frac{1}{1-z^{-1}\lambda_k}}
   = \prod_{k=1}^{\infty} 
      \abs{1+ \frac{\lambda_k}{z-\lambda_k}} \\
  \leq \prod_{k=1}^{\infty}
     \left( 1+ \frac{\abs{\lambda_k}}{\abs{z-\lambda_k}}\right)
  \leq \prod_{k=1}^{\infty} \left( 1+ \frac{\abs{\lambda_k}}{d(z,
         \sigma(A))} \right) 
   \leq \prod_{k=1}^{\infty} \left( 1+ \frac{s_k}{d(z,
       \sigma(A))}\right )\,,
\end{multline*}
where the last inequality follows from (\ref{vv}). 
\end{proof}

We now provide an upper bound for $\det(I-z^{-1}A)$ whenever $z$ is an
eigenvalue of an operator $B\in L(H)$ expressible in terms of the
distance of $z$ to $\sigma(A)$ and the distance of $A$ and $B$. 

\begin{proposition}
\label{prop:detu}
Suppose that $\dim H=\infty$ and that 
$A,B\in L(H)$. If $A$ has finite rank, then for any $z\in \rho(A)\cap
\sigma_p(B)$ we have  
  \[ \abs{\det(I-z^{-1}A)}
     \leq \frac{\norm{A-B}}{d(z,\sigma(A))}
           \prod_{k=1}^{n}\left(1+\frac{s_k(A)}{d(z,\sigma(A))}\right)\,,\]
where $n=\rank(A)$. 
\end{proposition}
\begin{proof}
Fix $z\in \rho(A)\cap \sigma_p(B)$. Note that since $\dim H=\infty$, 
we have $0\in \sigma(A)$, so $z\neq 0$. 
Let $e$ be an eigenvector of $B$ corresponding to $z$ and let 
$E$ denote the smallest closed linear span containing the range of $A$ and $e$. 
Clearly, $E$ is an invariant subspace of $A$ with $\nu :=\dim E \leq n+1$. 
Writing $T_E$ for the restriction of
an operator $T\in L(H)$ to $E$, we have 
\[ 
\lambda_k(A)=
\begin{cases}
  \lambda_k(A_E)  & \text{if $k\leq \nu$,} \\
       0          & \text{if $k>\nu$.}
\end{cases}
\]
Thus 
\begin{equation}
\label{eq:detu1} 
\abs{\det (I-z^{-1}A)}
   = \prod_{k=1}^{\nu}\abs{\lambda_k(I_E-z^{-1}A_E)} 
   \leq \prod_{k=1}^{\nu} s_k(I_E-z^{-1}A_E)\,,
 \end{equation}  
where we have used the multiplicative Weyl 
inequality\footnote{In fact there is equality
here, since for any $\nu\times \nu$ matrix $M$ we have 
\[ \prod_{k=1}^\nu\abs{\lambda_k(M)}^2=\abs{\det(M)}^2=\det(M^*M)=
\prod_{k=1}^\nu s_k(M)^2\,.\]
}
(\ref{pp}).

Now, for $k< \nu $, we have the bound 
\begin{equation}
\label{eq:detu2}
s_k(I_E-z^{-1}A_E)\leq 1+\abs{z}^{-1}s_k(A_E) \leq
1+\abs{z}^{-1}s_k(A)\,,
\end{equation}
which follows from (\ref{eq:snisan}). 
In order to treat the $\nu$-th factor define $K:E\to E$ by 
$K=I_E-z^{-1}PB_E$, where $P$ denotes the orthogonal projection onto $E$. 
Since $Ke=0$ we have $\rank K < \nu$, so 
\begin{multline}
\label{eq:detu3}
s_\nu(I_E-z^{-1}A)
\leq \norm{(I_E-z^{-1}PA_E)-K}\\
=\abs{z}^{-1}\norm{PB_E-PA_E}
\leq \abs{z}^{-1}\norm{A-B}\,,
\end{multline}
where we have used (\ref{eq:snisan}) and (\ref{eq:snpag}). 
Combining (\ref{eq:detu1}), (\ref{eq:detu2}), and (\ref{eq:detu3})
and taking into account that $\abs{z}\geq d(z,\sigma(A))$, 
since $0\in \sigma(A)$, we have 
\[ \abs{\det(I-z^{-1}A)}
     \leq \frac{\norm{A-B}}{d(z,\sigma(A))}
           \prod_{k=1}^{\nu-1}\left(1+\frac{s_k(A)}{d(z,\sigma(A))}\right)\,,\]
and the assertion follows. 
\end{proof}

\begin{remark}
\label{rem:detu}
Note that in the above proof we have only used the hypothesis $\dim H=\infty$
to conclude that $\abs{z}\geq
d(z,\sigma(A))$. Inspection of the proof shows that the
following finite-dimensional analogue of the above result holds. 

Let $A,B\in L(\C^n)$. Then for any $z\in \rho(A)\cap
\sigma(B)$ with $z\neq 0$ we have  
  \[ \abs{\det(I-z^{-1}A)}
     \leq \frac{\norm{A-B}}{d(z,\sigma(A)\cup \{0\})}
           \prod_{k=1}^{n-1}\left(1+\frac{s_k(A)}{d(z,\sigma(A)\cup
               \{0\})}\right)\,.\]
Here we have used the fact that both the eigenvector of $B$ and the
range of $A$ trivially belong to the same $n$-dimensional space. 
\end{remark}

The above proposition can be extended to arbitrary trace class
operators using a standard approximation procedure. 

\begin{proposition}
\label{prop:detug}
Let $H$ be infinite-dimensional and suppose that $A\in S_1(H)$ and 
$B\in L(H)$. Then for any $z \in
\rho(A) \cap \sigma_p(B)$ we have  
  \[ \abs{\det(I-z^{-1}A)}\leq \frac{\norm{A-B}}{d(z,
    \sigma(A))}F_A
   \left ( \frac{1}{d(z,\sigma(A))}
       \right )\,.\] 
\end{proposition}
\begin{proof}
Since $A$ is trace class it has a Schmidt representation 
of the form  
\[ Ax= \sum_{k=1}^{\infty}s_k{(A)}(x, a_k)b_k \quad (\forall x \in H)\,,\]
where $(a_k)^{\infty}_{k = 1}$ and $(b_k)^{\infty}_{k = 1}$ are
orthonormal systems in $H$ (see \cite[Chapter~VI, Theorem~1.1]{3}).  
For any $n\in \N$ we now define a finite-rank operator $A_n$ by 
\[ A_nx= \sum_{k=1}^{n}s_k{(A)}(x, a_k)b_k\, \quad (\forall x \in H)\,.\]
A short calculation shows that 
\begin{equation}
\label{eq:snkA}
  s_k(A_n) = 
  \begin{cases}
     s_k(A) & \text{ for $k\leq n$,}\\
       0    & \text{ for $k>n$,}
  \end{cases}
\end{equation}   
and that 
\[ \lim_{n\to \infty}\norm{A_n-A}_1=0\,. \] 
Now fix $z\in \rho(A) \cap \sigma_p(B)$. 
Since $z\in \rho(A)$ standard perturbation 
theory 
implies that 
there is $N\in \N$ such that $z\in \rho(A_n)$ for all $n\geq N$ 
(see, for example, \cite[Chapter~II, Theorem~4.1]{3}). 
Thus, by Proposition~\ref{prop:detu}, we have  
\begin{equation}
\label{eq:detbound}
   \abs{\det(I-z^{-1}A_n)}\leq 
   \frac{\norm{A_n-B}}{d(z,
    \sigma(A_n))}F_{A}
   \left ( \frac{1}{d(z,\sigma(A_n))}
       \right )\quad (\forall n \geq N)\,,
\end{equation}
where we have used the fact that $F_{A_n}(r)\leq F_A(r)$ for every $r\in
[0,\infty)$, which follows from  (\ref{eq:snkA}). 

In order to conclude the proof we make a number of observations. 
Since the determinant is Lipschitz-continuous on the unit ball of
$S_1(H)$ (see, for example, \cite[Theorem~6.5]{simon}) we have  
\[ \lim_{n \to \infty}\det(I-z^{-1}A_n)=\det(I-z^{-1}A)\,. \]
Furthermore, since the spectrum of $A$ is discrete, the spectrum of $A_n$ converges 
to the spectrum of $A$ in the Hausdorff metric (see \cite[Theorem~3]{newburgh}), so 
\[ \lim_{n\to \infty}d(z,\sigma(A_n)=d(z,\sigma(A))\,. \]
Finally, we clearly have
\[ \lim_{n\to \infty}\norm{A_n-B}=\norm{A-B}\,.  \]
The assertion now follows by taking the limit on both sides of (\ref{eq:detbound}). 
\end{proof}


\section{Bounds for the spectral distance}
\label{sec:4}

The upper and lower bounds for the determinants obtained in the previous section 
can be combined to produce quantitative upper bounds for the spectral distance of 
two trace class operators. Before we proceed we require some more notation. 

Let $\mathcal{F}$ denote the set 
of all strictly monotonically increasing functions 
$F:[0,\infty]\to [0,\infty)$ with $\lim_{r\to \infty}F(r)=\infty$, and let 
$\mathcal{F}_0$ denote the subset of those $F\in \mathcal{F}$ with $F(0)=0$.  
Both $\F$ and $\F_0$ are partially ordered by defining 
\[ F_1\leq F_2 :\Leftrightarrow F_1(r)\leq F_2(r) \quad (\forall r \geq 0)\,.\]
On $\mathcal F$ we define an operation $H$, which will allow us to  
express our bounds in a succinct way: 
\[ H:\mathcal{F} \to  \mathcal{F}_0 \]  
\[ F\mapsto H_F\,,\]
where $H_F$ is given by
\[ H_F(r)=\frac{1}{\tilde{F}^{-1} \left ( \frac{1}{r} \right )}\,, \]
and $\tilde{F}^{-1}$ is the inverse of
$\tilde{F}\in \F_0$ defined 
by $\tilde{F}(r)=rF(r)^2$. 

For later use, we note the following simple consequences of the definition of $H$, 
the proofs of which are left as an exercise. 

\begin{lemma} 
\label{lem:Hlem}
Let $F_1, F_2\in \F$ and, for $m>0$, let $M_m\in \F$ be given 
by $M_m(r)=mr$. Then the following holds.  
  \begin{enumerate}
  \item \label{eq:Hlem1}If $F_1\leq F_2$ then $H_{F_1}\leq H_{F_2}$.
  \item \label{eq:Hlem2}If $F_1=F_2\circ M_m$, then $H_{F_1}=M_m\circ H_{F_2}\circ M_m^{-1}$.  \end{enumerate}
\end{lemma}

We are now able to state and prove our main results.
\begin{theorem}
\label{thm:dhat}
Suppose that $\dim H=\infty$ and that $A\in S_1(H)$. 
Then
\[  \hat{d}(\sigma_p(B), \sigma(A)) \leq H_{F_A}(\norm{A-B}) 
  \quad (\forall B \in L(H))\,. \] 
\end{theorem}
\begin{proof}
Fix $B\in L(H)$. We need to show that 
\[ d(z,\sigma(A)) \leq H_{F_A}(\norm{A-B}) \quad (\forall z \in
\sigma_p(B))\,.\]
For $z \in \sigma(A)$ the inequality above is trivially satisfied, so
we shall assume that $z\in \rho(A)\cap \sigma_p(B)$. Noting that $z\neq 0$ (since 
$\dim H=\infty$), Propositions~\ref{prop:detl} and \ref{prop:detug}
yield 
\[ \frac{1}{F_A\left(\frac{1}{d(z, \sigma(A))}\right)}
  \leq \abs{ \det(I-z^{-1}A)} 
  \leq \frac{\norm{A-B}}{d(z, \sigma(A))} F_A\left(\frac{1}{d(z, \sigma(A))} \right)\,.\]
Hence  
\[\frac{1}{\norm{A-B}}\leq\tilde{F}_A \left(\frac{1}{d(z, \sigma(A))} \right)\,,\]
where $\tilde{F}_A(r)=rF_A(r)^2$, 
so 
\[\tilde{F}_A^{-1}\left(\frac{1}{\norm{A-B}}\right)\leq
\frac{1}{d(z, \sigma(A))}\,, \]
which implies 
\[d(z, \sigma(A)) 
\leq \frac{1}{\tilde{F}_A^{-1}\left(\frac{1}{\norm{A-B}}\right)}
=H_{F_A}(\norm{A-B})\,,\]
and the assertion follows.  
\end{proof}
An immediate consequence of the previous theorem is the following quantitative bound 
for the Hausdorff distance of the spectra of two trace class operators.  
\begin{theorem} 
\label{thm:hdist}
Suppose that $\dim H=\infty$ and that  $A,B\in S_1(H)$.  
Then 
\[  \Hdist(\sigma(A), \sigma(B)) 
    \leq H_{F_{A,B}}(\norm{A-B})\,,\]
where $F_{A,B}\in \F$ is given by   
\[ F_{A,B}(r)=\max\{ F_A(r), F_B(r) \}\,.\]
\end{theorem} 

\begin{proof}
Using Theorem~\ref{thm:dhat} together with Lemma~\ref{lem:Hlem} and the fact 
that $B$ is trace class we have 
\[ \hat{d}(\sigma(B), \sigma(A)) 
  =\hat{d}(\sigma_p(B), \sigma(A)) 
  \leq H_{F_{A,B}}(\norm{A-B})\,.  \]
But by symmetry we also have 
\[ \hat{d}(\sigma(A), \sigma(B))  
  \leq H_{F_{A,B}}(\norm{A-B})\,,  \]
and the assertion follows. 
\end{proof}

\begin{remark}
\label{rem:fdim}
The condition $\dim H=\infty$ 
is not essential in the previous two 
theorems. Inspection of their 
proofs together with Remark~\ref{rem:detu} yield the following 
finite-dimensional analogue of Theorem~\ref{thm:dhat} and \ref{thm:hdist}. 

Let $A,B\in L(\C^n)$. Then
\[  \hat{d}(\sigma(B),\sigma(A)\cup\{0\}) 
    \leq H_{F_{A}}(\norm{A-B})\,,\]
and 
\[  \Hdist(\sigma(A)\cup\{0\}, \sigma(B)\cup\{0\}) 
    \leq H_{F_{A,B}}(\norm{A-B})\,.\]
\end{remark}

\section{Comparison with other bounds}
\label{sec:5}
In order to apply the bounds obtained in the previous section some knowledge of the 
singular values of the operators involved is required. Lemma~\ref{lem:Hlem} implies 
that any upper bound for the singular values translates into a bound for 
the spectral distance of two trace class operators. We shall briefly 
discuss a number of possibilities, 
which either reproduce or improve existing bounds. 

We start with a bound involving only one free parameter, namely the trace norm. 
Given $A\in S_1(H)$, we have 
\[ F_A(r)=\prod_{k=1}^\infty(1+rs_k(A))\leq \exp( r\norm{A}_1)\,.\]
Theorem~\ref{thm:hdist} and Lemma~\ref{lem:Hlem} immediately give the following result, 
reproducing the known bound from \cite[Theorem~5.2]{1} obtained by a different method. 

\begin{corollary}
\label{coro:trace}
Let $\dim(H)=\infty$. Then, for any $A,B\in S_1(H)$ not both of them the zero operator, 
we have 
\[ \Hdist(\sigma(A),\sigma(B))\leq m H_{G^S}\left ( \frac{\norm{A-B}}{m} \right )\,, \]
where $G^S(r)=\exp(r)$ and $m=\max\{\norm{A}_1, \norm{B}_1\}$. 
\end{corollary}
 
Another possibility is to demand that the singular values decay at 
a stretched exponential rate. This leads to the notion of exponential
classes introduced in \cite{10}. 

\begin{definition}
Let $a>0$ and $\alpha>0$. Then 
\[ E(a,\alpha;H)=\all{A\in S_1(H)}{ |A|_{a,\alpha}:=\sup_{k\in\N}
  s_k(A)\exp(ak^{\alpha}) < \infty}\,, \]
is called the \emph{exponential class of type $(a,\alpha)$}. The number 
$|A|_{a,\alpha}$ is called the \emph{$(a,\alpha)$-gauge} or simply 
\emph{gauge} of $A$. 
\end{definition}

Naturally occurring operators belonging to
 these classes include integral operators with analytic kernels (see
 \cite{KR}), composition operators with strictly contracting holomorphic symbols,
or, more generally, 
transfer operators arising from real analytic expanding maps, 
which play an important 
role in smooth ergodic  theory (see \cite{bj_advances}).  
Certain operators on spaces of harmonic functions also belong to this class 
(see \cite{bandtlow_chu}). 

Specialising to operators belonging to exponential classes, 
Theorem~\ref{thm:hdist} and Lemma~\ref{lem:Hlem} yield the following. 

\begin{corollary}
\label{coro:expo}
Let $\dim H=\infty$. Then, for any $A,B\in E(a,\alpha;H)$ not both of them the 
zero operator, we have
\[ \Hdist(\sigma(A),\sigma(B))\leq m H_{G^E_{a,\alpha}}
   \left ( \frac{\norm{A-B}}{m} \right )\,, \]
where 
\[ G^E_{a,\alpha}(r)=\prod_{k=1}^\infty(1+r\exp(-an^\alpha)) \]
and $m=\max\{\abs{A}_{a,\alpha}, \abs{B}_{a,\alpha}\}$.  
\end{corollary}
By \cite[Proposition~3.1]{10} we have 
\[ 
\log G^E_{a,\alpha}(r) \sim a^{1/\alpha} \frac{\alpha}{1+\alpha}(\log r)^{1+1/\alpha} 
\quad \text{as $r\to \infty$} 
\]
and a short calculation shows that 
\[ \log H_{G^E_{a,\alpha}}(r) \sim -(2a)^{1/(1+\alpha)} 
 \left ( \frac{1+\alpha}{\alpha} \right )^{\alpha/(1+\alpha)}\abs{\log r}^{\alpha/(1+\alpha)} 
\quad \text{as $r\downarrow 0$}\,, \]
thus Corollary~\ref{coro:expo} improves the 
bound in \cite[Theorem~4.2]{10} obtained using a different method. 

Finally, we turn to the most stringent condition for the decay of the singular 
values, which arises when all but finitely many of the singular values
vanish, or, in other words, if the operators involved are finite
rank. 

Using the same argument as in the previous two cases we obtain a bound
that only relies on knowledge of the first singular value, that is, 
the operator norm. For better comparison with Elsner's bound we
formulate it for $n\times n$ matrices (see Remark~\ref{rem:fdim}). 

\begin{corollary}
\label{coro:fr1}
Let $A,B\in L(\C^n)$. Then  
\[ \Hdist(\sigma(A)\cup\{0\},\sigma(B)\cup\{0\})\leq m H_{G^F}
   \left ( \frac{\norm{A-B}}{m} \right )\,, \]
where 
\[ G^F(r)=(1+r)^n \]
and $m=\max\{\norm{A}, \norm{B}\}$.  
\end{corollary}
Note that the bound in the above corollary is of a form very similar to 
Elsner's bound, which for $A,B\in L(\C^n)$ can be written 
\[ \Hdist(\sigma(A),\sigma(B))\leq m H
   \left ( \frac{\norm{A-B}}{m} \right )\,, \]
where $m=\norm{A}+\norm{B}$ and $H(r)=r^{1/n}$. 
Clearly, the bound in Corollary~\ref{coro:fr1} cannot compete with
the Elsner bound (except possibly in special cases), since  
\[ H_{G^F}(r) \sim r^{1/(2n+1)} \quad \text{as $r\downarrow 0$}\,.  \] 

However, in deriving the bound above we have only used the first
singular value. If we have information about the first two singular
values, the bound can be improved. 

\begin{corollary}
\label{coro:fr2}
Let $A,B\in L(\C^n)$. Then  
\[ \Hdist(\sigma(A)\cup\{0\},\sigma(B)\cup\{0\})
   \leq H_{G_{s_1,s_2}^F}
    (\norm{A-B} )\,, \]
where 
\[ G_{s_1,s_2}^F(r)=(1+rs_1)(1+rs_2)^{n-1}\,, \]
with $s_1=\max\{s_1(A), s_1(B)\}$ and $s_2=\max\{s_2(A), s_2(B)\}$. 
\end{corollary}
A short calculation shows that 
\[   
H_{G_{s_1,s_2}^F}(r)\sim
s_1^{2/(2n+1)}s_2^{(2n-2)/(2n+1)}r^{1/(2n+1)} 
\quad \text{as $r\downarrow 0$}. \]
In order to see that the bound in Corollary~\ref{coro:fr2} can be
better than Elsner's bound we define weighted shift 
matrices $A_\epsilon, B_\epsilon\in
L(\C^n)$ by letting 
\[
A_\epsilon e_k=
\begin{cases}
           e_2    & \text{if $k=1$,}\\
  \epsilon e_{k+1} & \text{if $1<k<n$,}\\
           0      & \text{if $k=n$,}
\end{cases}
\quad 
B_\epsilon e_k=
\begin{cases}
           e_2    & \text{if $k=1$,}\\
  \epsilon e_{k+1} & \text{if $1<k<n$,}\\
  \epsilon e_1    & \text{if $k=n$,}
\end{cases}
\]
where $(e_k)_{k=1}^n$ is an orthonormal basis of $\C^n$. Then we have 
\[ \norm{A_\epsilon-B_\epsilon}=\epsilon\,, \]
\[ s_1(A_\epsilon)=s_1(B_\epsilon)=1\,, \]
and  
\[ s_2(A_\epsilon)\leq s_2(B_\epsilon)=\epsilon\,, \] 
provided that $n>1$. 
For this family of matrices Elsner's bound gives 
\[ H(\norm{A_\epsilon-B_\epsilon})\sim
  \epsilon^{1/n} \quad \text{as $\epsilon \downarrow 0$}\,, \]
which is worse than what we get from the bound in
Corollary~\ref{coro:fr2}, namely 
\[ 
H_{G_{1,\epsilon}^F}(\norm{A_\epsilon-B_\epsilon})\sim
  \epsilon^{(2n-1)/(2n+1)} 
\quad \text{as $\epsilon \downarrow 0$}\,. 
\]
In other words, our bound fares better then Elsner's for matrices
which are small perturbations of matrices of rank one. Of
course, this advantage is obtained by requiring more information about
the matrices, namely upper bounds for the first and second singular
values. 

The previous discussion should have elucidated the following general principle 
at work here. 
Upper bounds for the singular values of $A$ and $B$ translate into 
upper bounds for the function $H_{F_{A,B}}$, which bounds the spectral variation of 
$A$ and $B$. Moreover, the faster the bounds for the singular values decay to zero, the 
slower $F_{A,B}$ will grow at infinity, and so, the faster $H_{F_{A,B}}$ will tend 
to zero at zero. 

\section{Acknowledgements}
We are very grateful to an anonymous referee for bibliographic assistance. 
OFB would like to thank Henk Bruin at the University 
of Vienna for his kind hospitality from May until August 2014, during
which period 
this article was finalised. AG expresses her 
gratitude to Wolfram Just for enlightening discussions during the preparation of this
article.

\end{document}